\documentclass[12pt]{amsart}
\usepackage{amsmath, amsthm, amscd, amsfonts, amssymb, graphicx, color}
\usepackage[bookmarksnumbered, colorlinks, plainpages]{hyperref}
\usepackage[cp1251]{inputenc}
\usepackage[T2A]{fontenc}
\usepackage[english]{babel}
\usepackage{amsmath,amsfonts,amssymb}
\usepackage{geometry}
\numberwithin{equation}{section}

\newtheorem{theorem}{Theorem}[section]

\newtheorem{proposition}[theorem]{Proposition}
\theoremstyle{remark}
\newtheorem{remark}[theorem]{Remark}
\theoremstyle{definition}

\newtheorem{example}[theorem]{Example}

\title[Interpolation Hilbert spaces between Sobolev spaces]
{Interpolation Hilbert spaces \\ between Sobolev spaces}

\author[V. A. Mikhailets, A. A. Murach]{Vladimir A. Mikhailets, Aleksandr A. Murach}


\thanks{This research was supported by grant no. 03-01-12 of National Academy of Sciences of Ukra\-ine (under the joint Ukrainian--Russian project of NAS of Ukraine and the Siberian Branch of Russian Academy of Sciences).}


\begin{document}

\maketitle

\begin{abstract}
We explicitly describe all Hilbert function spaces that are interpolation spaces
with respect to a given couple of Sobolev inner product spaces considered over
$\mathbb{R}^{n}$ or a half-space in $\mathbb{R}^{n}$ or a bounded Euclidean domain
with Lipschitz boundary. We prove that these interpolation spaces form a subclass of
isotropic H\"ormander spaces. They are parametrized with a radial function parameter
which is OR-varying at $+\infty$ and satisfies some additional conditions. We give
explicit examples of intermediate but not interpolation spaces.
\end{abstract}

\section{Introduction}\label{sec1}

The fundamental importance of Sobolev spaces for analysis and the theory of partial differential equations is well known. This importance comes partly from interpolation properties of the Sobolev scale \cite{LionsMagenes72,Triebel95}. Owing
to them, it is possible to extend significant properties of integer order Sobolev
spaces to fractional order spaces. The invariance of spaces with respect to
admissible change of variables is one of these properties and enables the Sobolev
spaces (of any real order) to be well-defined over smooth manifolds (see, e.g., \cite[Sec.~2.6]{Hermander63}).

Nevertheless, the Sobolev scale is not sufficiently fine for a number of
mathematical problems, which explains a natural need to replace the Sobo\-lev spaces
with more general isotropic H\"ormander spaces $H^{\varphi}$
\cite{Hermander63,Hermander83}. They are initially defined over $\mathbb{R}^{n}$
with the help of the Fourier transform and a radial weight function $\varphi$ of the
scalar argument $\langle\xi\rangle:=(1+|\xi|^{2})^{1/2}\geq1$. The use of the power
function $\varphi(t)=t^{s}$ leads to the Sobolev space $H^{(s)}$, with
$s\in\mathbb{R}$. The application of interpolation with a function parameter allows
the classical theory of elliptic boundary-value problems to be completely
transferred from the Sobolev scale to a more extensive class of H\"ormander inner
product spaces (see
\cite{06UMJ3,MikhailetsMurach14,12BJMA2} and references
therein). The latter---refined Sobolev scale---is parametrized with a function
$\varphi$ that varies regularly in J.~Karamata's sense at $+\infty$ with an
arbitrary index $s\in\mathbb{R}$.

Interpolation with a function parameter between spaces of differentiable functions has investigated less completely \cite{CobosFernandez88,Merucci84,Schechter67} than the interpolation between spaces of integrable functions. The purpose of this paper is to describe constructively all Hilbert function spaces that are interpolation spaces with respect to a given couple of Sobolev inner product spaces
\begin{equation}\label{eqa}
\bigl[H^{(s_{0})},H^{(s_{1})}\bigr],\quad-\infty<s_{0}<s_{1}<+\infty,
\end{equation}
considered over $\mathbb{R}^{n}$ or a half-space in $\mathbb{R}^{n}$ or a bounded Euclidean domain. We show that these interpolation spaces form a subclass of
isotropic H\"ormander spaces. This subclass is characterized by the fact that the
function $\varphi$ is OR-varying at $+\infty$ in the sense of V.~G.~Avakumovi\'c
\cite{BinghamGoldieTeugels89,Seneta76} and satisfies two additional conditions,
which involve the numbers $s_{0}$ and $s_{1}$ respectively. The chosen class of
function spaces is broad enough to be effectively used in various problems of the modern analysis \cite[Sec.~2.4]{MikhailetsMurach14}. In particular, we may define the largest class of H\"ormander inner product spaces over a closed compact manifold in such a way that they do not depend on a choice of local charts but possess the interpolation property. These results and their applications will be given in a separate paper.

This paper is organized in the following way. The main result of the paper is
formulated as Theorems \ref{th1} and \ref{th2} in the next section \ref{sec2}. Here
we also define necessary classes of function parameters and Hilbert function spaces.
Section~\ref{sec3} contains auxiliary results. Here we discuss the interpolation
with a function parameter of abstract Hilbert spaces, which is the main tool in our
proofs. The main result is proved in Section~\ref{sec4}. In Section~\ref{sec5} we
investigate the class of all Hilbert spaces that are interpolation spaces between
Sobolev inner product spaces. We show that this class is the complete extension of
the scale $\{H^{(s)}\mid s\in\mathbb{R}\}$ by the interpolation within the category
of Hilbert spaces. In the final section \ref{sec6}, we construct explicit and
mutually complementary examples of Hilbert spaces that are intermediate but not
interpolation spaces for couples of Sobolev spaces. We were not able to find any
examples of this type in the literature.

\section{H\"ormander spaces and the main results}\label{sec2}

Suppose that $\mu:\mathbb{R}^{n}\rightarrow(0,\infty)$ is a Borel measurable weight
function in the following sense: there exist numbers $c\geq1$ and $\ell\geq0$ such that
\begin{equation}\label{eq2}
\frac{\mu(\xi)}{\mu(\eta)}\leq c\,(1+|\xi-\eta|)^{\ell}\quad\mbox{for
arbitrary}\quad\xi,\eta\in\mathbb{R}^{n}.
\end{equation}

By definition (see \cite[Sec.~2.2]{Hermander63} or \cite[Sec.~2]{VolevichPaneah65}), the H\"ormander space $H^{\mu}_{2}(\mathbb{R}^{n})$ consists of all
L.~Schwartz's distributions $u\in\mathcal{S}'(\mathbb{R}^{n})$ that their Fourier
transform $\widehat{u}$ is locally Lebesgue integrable over $\mathbb{R}^{n}$ and
$\mu\,\widehat{u}\in\mathrm{L}_{2}(\mathbb{R}^{n})$. The norm in the complex linear
space $H^{\mu}_{2}(\mathbb{R}^{n})$ is defined by the formula
\begin{equation*}
\|u\|_{H^{\mu}_{2}(\mathbb{R}^{n})}:=
\|\mu\,\widehat{u}\|_{\mathrm{L}_{2}(\mathbb{R}^{n})}.
\end{equation*}

The space $H^{\mu}_{2}(\mathbb{R}^{n})$ is Hilbert with respect to this norm and is
continuously embedded in $\mathcal{S}'(\mathbb{R}^{n})$. This space is separable and
$C^{\infty}_{0}(\mathbb{R}^{n})$ is dense in it. Note that
\begin{equation}\label{f2.2}
H^{\mu}_{2}(\mathbb{R}^{n})\subseteq
H^{\nu}_{2}(\mathbb{R}^{n})\quad\Longleftrightarrow\quad (\nu/\mu\;\;\mbox{is
bounded on}\;\;\mathbb{R}^{n});
\end{equation}
here $\mu$ and $\nu$ are weight functions, and the embedding is continuous.

Among the H\"ormander spaces $H^{\mu}_{2}(\mathbb{R}^{n})$ we only need isotropic
spaces. They correspond to radial weight functions
$\mu(\xi)=\varphi(\langle\xi\rangle)$, where $\varphi$ belongs to the following
class of function parameters.

Let $\mathrm{OR}$ be the class of all Borel measurable functions
$\varphi:\nobreak[1,\infty)\rightarrow(0,\infty)$ for which there exist numbers
$a>1$ and $c\geq1$ such that
\begin{equation}\label{eq3}
c^{-1}\leq\frac{\varphi(\lambda t)}{\varphi(t)}\leq c\quad\mbox{for arbitrary}\quad
t\geq1\quad\mbox{and}\quad\lambda\in[1,a],
\end{equation}
with $a$ and $c$ depending on $\varphi$. Such functions are said to be OR-varying at
$+\infty$. This function class was introduced by V.~G.~Avakumovi\'c in 1936 and has
been comprehensively investigated \cite{BinghamGoldieTeugels89,Seneta76}. We recall some
of its known properties (see, e.g., \cite[Sec.~A.1]{Seneta76}).

\begin{proposition}\label{prop1}
\begin{enumerate}
\item[$\mathrm{(i)}$] If $\varphi\in\mathrm{OR}$, then both the functions $\varphi$ and
$1/\varphi$ are bounded on every compact interval $[1,b]$ with $1<b<\infty$.
\item[$\mathrm{(ii)}$] The following description of the class $\mathrm{OR}$ holds:
\begin{equation*}
\varphi\in\mathrm{OR}\quad\Longleftrightarrow\quad\varphi(t)=
\exp\left(\beta(t)+\int\limits_{1}^{t}\frac{\varepsilon(\tau)}{\tau}\;d\tau\right),
\;\;t\geq1,
\end{equation*}
where the real-valued functions $\beta$ and $\varepsilon$ are Borel measurable and
bounded on $[1,\infty)$.
\item [$\mathrm{(iii)}$] For an arbitrary function
$\varphi:[1,\infty)\rightarrow(0,\infty)$, condition \eqref{eq3} is equivalent to
the following: there exist numbers $s_{0},s_{1}\in\mathbb{R}$, $s_{0}\leq s_{1}$,
and $c\geq1$ such that
\begin{equation}\label{eq4}
\begin{gathered}
\frac{\varphi(t)}{t^{s_{0}}}\leq
c\,\frac{\varphi(\tau)}{\tau^{s_{0}}}\quad\mbox{and}\quad
\frac{\varphi(\tau)}{\tau^{s_{1}}}\leq
c\,\frac{\varphi(t)}{t^{s_{1}}}\\
\mbox{for all}\quad t\geq1\quad\mbox{and}\quad\tau\geq t.
\end{gathered}
\end{equation}
\end{enumerate}
\end{proposition}

Condition \eqref{eq4} means that the function $\varphi(t)/t^{s_{0}}$ is equivalent to an increasing function, whereas the function $\varphi(t)/t^{s_{1}}$ is equivalent to a decreasing one on $[1,\infty)$. Here and below we say that positive functions $\psi_{1}$ and $\psi_{2}$ are equivalent on a given set if both $\psi_{1}/\psi_{2}$ and $\psi_{2}/\psi_{1}$ are bounded on it; this property will be denoted by $\psi_{1}\asymp\psi_{2}$.

Setting $\lambda:=\tau/t$ in condition \eqref{eq4} we rewrite it in the equivalent form
\begin{equation}\label{eq5}
c^{-1}\lambda^{s_{0}}\leq\frac{\varphi(\lambda t)}{\varphi(t)}\leq
c\lambda^{s_{1}}\quad\mbox{for all}\quad t\geq1\quad\mbox{and}\quad\lambda\geq1.
\end{equation}
We associate the following notation with $\varphi\in\mathrm{OR}$:
\begin{gather}\label{eq6}
\sigma_{0}(\varphi):=\sup\{s_{0}\in\mathbb{R}\mid\mbox{the left-hand inequality in
\eqref{eq5} holds}\},\\ \label{eq7}
\sigma_{1}(\varphi):=\inf\{s_{1}\in\mathbb{R}\mid\mbox{the right-hand inequality in
\eqref{eq5} holds}\}.
\end{gather}
Evidently, $-\infty<\sigma_{0}(\varphi)\leq\sigma_{1}(\varphi)<\infty$. The numbers
$\sigma_{0}(\varphi)$ and $\sigma_{1}(\varphi)$ are equal to the lower and the upper
Matuszewska indices of $\varphi$, respectively (see \cite{Matuszewska64} and \cite[Theorem 2.2.2]{BinghamGoldieTeugels89}). In particular, if $\varphi$ is regularly varying with an index $s\in\mathbb{R}$, then $\varphi\in\mathrm{OR}$ and $\sigma_{0}(\varphi)=\sigma_{1}(\varphi)=\nobreak s$ (see \cite[Sec. 1.4.2]{BinghamGoldieTeugels89} or \cite[Sec. 1.1]{Seneta76}).

\begin{remark}
The Matuszewska indices are closely connected with the classical Boyd indices \cite{Boyd67}. Let us clarify this connection. Given $\varphi\in\mathrm{OR}$, we define $\varphi(t):=\varphi^{2}(1)/\varphi(t^{-1})$ for $0<t<1$ and obtain a positive function $\varphi$, which is given on $(0,\infty)$ and satisfies the condition
$$
c^{-2}\lambda^{s_{0}}\leq\frac{\varphi(\lambda t)}{\varphi(t)}\leq
c^{2}\lambda^{s_{1}}\quad\mbox{for all}\quad t>0\quad\mbox{and}\quad\lambda\geq1
$$
by virtue of \eqref{eq5}. Consider the function
$$
m_{\varphi}(\lambda):=\sup_{t>0}\,\frac{\varphi(\lambda t)}{\varphi(t)}\quad\mbox{of}\quad\lambda>0.
$$
Then \cite[Chap.~II, \S~1, Subsec.~2]{KreinPetuninSemenov82}
$$
\sigma_{0}(\varphi)= \lim_{\lambda\rightarrow0+}\,\frac{\log
m_{\varphi}(\lambda)}{\log\lambda},\quad \sigma_{1}(\varphi)=
\lim_{\lambda\rightarrow\infty}\,\frac{\log m_{\varphi}(\lambda)}{\log\lambda}.
$$
By definition, the right-hand sides of these equalities are the Boyd indices of the function $m_{\varphi}(\lambda)$.
\end{remark}

Let $\varphi\in\mathrm{OR}$. By definition, $H^{\varphi}(\mathbb{R}^{n})$ is the
Hilbert space $H^{\mu}_{2}(\mathbb{R}^{n})$ with
$\mu(\xi):=\varphi(\langle\xi\rangle)$ for all $\xi\in\mathbb{R}^{n}$. The inner
product in $H^{\varphi}(\mathbb{R}^{n})$ is
\begin{equation*}
(u_{1},u_{2})_{H^{\varphi}(\mathbb{R}^{n})}:=
\int\limits_{\mathbb{R}^{n}}\varphi^{2}(\langle\xi\rangle)
\,\widehat{u_{1}}(\xi)\,\overline{\widehat{u_{2}}(\xi)}\,d\xi.
\end{equation*}
It induces the norm introduced above. Note that the space $H^{\varphi}(\mathbb{R}^{n})$ is well-defined, since the function $\mu(\xi)=\varphi(\langle\xi\rangle)$ of $\xi\in\nobreak\mathbb{R}^{n}$ satisfies \eqref{eq2}. This will be demonstrated in Proposition~\ref{prop7}, Section~\ref{sec3}.

We also introduce necessary function spaces over Euclidean domains. Let $\Omega$ be
a domain in~$\mathbb{R}^{n}$. By definition, the linear space $H^{\varphi}(\Omega)$
consists of the restrictions $v=u\!\upharpoonright\!\Omega$ to $\Omega$ of all
distributions $u\in H^{\varphi}(\mathbb{R}^{n})$. The norm in $H^{\varphi}(\Omega)$
is
\begin{equation*}
\|v\|_{H^{\varphi}(\Omega)}:=\inf\left\{\|u\|_{H^{\varphi}(\mathbb{R}^{n})}\mid u\in
H^{\varphi}(\mathbb{R}^{n}),\;\; u=v\;\;\mbox{in}\;\;\Omega\right\}.
\end{equation*}
The space $H^{\varphi}(\Omega)$ is a separable and Hilbert space with respect to the
above norm because it is a factor space of the separable Hilbert space
$H^{\varphi}(\mathbb{R}^{n})$ by
\begin{equation*}
\left\{w\in H^{\varphi}(\mathbb{R}^{n})\mid
\mathrm{supp}\,w\subseteq\mathbb{R}^{n}\setminus\Omega\right\}.
\end{equation*}

If $\varphi(t)=t^{s}$ with $t\geq1$ for some $s\in\mathbb{R}$, then
$H^{\varphi}(\Omega)$ coincides with the Sobolev space $H^{(s)}(\Omega)$ of order $s$ (we refer to the definition given in \cite[Sec. 4.2.1]{Triebel95}).

\begin{remark}\label{rem}
Let $\varphi,\chi\in\mathrm{RO}$. It follows from \eqref{f2.2} that
$H^{\varphi}(\Omega)=H^{\chi}(\Omega)$ if and only if $\varphi\asymp\chi$ on $[1,\infty]$. Specifically, the latter property yields $\sigma_{j}(\varphi)=\sigma_{j}(\chi)$ for every $j\in\{0,1\}$.
\end{remark}

Let $X:=[X_{0},X_{1}]$ be a couple of complex Hilbert spaces $X_{0}$ and $X_{1}$ that the continuous embedding $X_{1}\hookrightarrow X_{0}$ holds. (This assumption is caused by the fact that we are only interested in the case where $X$ is a couple of inner product Sobolev spaces.) A Hilbert space $H$ is called an interpolation space with respect to (or for) the couple $X$ if the following two conditions are satisfied:
\begin{enumerate}
\item [(i)] $H$ is an intermediate space for this couple; i.e., $X_{1}\subseteq H\subseteq X_{0}$, and the embeddings are continuous;
\item [(ii)] for every linear operator $T$ given on $X_{0}$, the following implication holds: if the restriction of $T$ to $X_{j}$ is a bounded operator on $X_{j}$ for each $j\in\{0,1\}$, then the restriction of $T$ to $H$ is a bounded operator on $H$.
\end{enumerate}

Property (ii) implies the following inequality for norms of the operators:
$$
\|T\|_{H\rightarrow H}\leq c\,\max\,\left\{\,\|T\|_{X_{0}\rightarrow
X_{0}},\;\|T\|_{X_{1}\rightarrow X_{1}}\,\right\},
$$
where $c$ is a positive number independent of $T$ (see \cite[Theorem
2.4.2]{BerghLefstrem76}). If this inequality holds with $c=1$ for every $T$, then the interpolation space $H$ is called exact.

Note that the above properties (i) and (ii) are invariant with respect to the choice
of an equivalent norm on $H$. Therefore, we will describe the interpolation spaces
up to equivalence of norms.

The main result of the paper consists of the following two theorems. Here and below
we suppose that $\Omega$ is either the whole space $\mathbb{R}^{n}$ or an open
half-space in $\mathbb{R}^{n}$ or a bounded domain in $\mathbb{R}^{n}$ with Lipschitz boundary.

\begin{theorem}\label{th1}
Let $-\infty<s_{0}<s_{1}<\infty$. A Hilbert space $H$ is an interpolation space with
respect to the couple of Sobolev spaces $[H^{(s_{0})}(\Omega),H^{(s_{1})}(\Omega)]$
if and only if $H=H^{\varphi}(\Omega)$ up to equivalence of norms for some function parameter $\varphi\in\mathrm{OR}$ that satisfies condition~\eqref{eq5}.
\end{theorem}

\begin{remark}\label{rem1}
Naturally, we mean in this theorem that the number $c\geq1$ in condition \eqref{eq5}
is independent of $t$ and $\lambda$. This condition is equivalent to the following pair
of conditions:
\begin{enumerate}
\item [$\mathrm{(i)}$] $s_{0}\leq\sigma_{0}(\varphi)$ and, moreover,
$s_{0}<\sigma_{0}(\varphi)$ if the supremum in $\eqref{eq6}$ is not attained;
\item [$\mathrm{(ii)}$] $\sigma_{1}(\varphi)\leq s_{1}$ and, moreover,
$\sigma_{1}(\varphi)<s_{1}$ if the infimum in $\eqref{eq7}$ is not attained.
\end{enumerate}
\end{remark}

\begin{remark}\label{rem2.4}
It is useful to note the following. Suppose that a Hilbert space $H$ is an
interpolation space with respect to a given couple of Sobolev spaces
$[H^{(s_{0})}(\Omega),H^{(s_{1})}(\Omega)]$. Then $H$ is an interpolation space for
each wider couple $[H^{(s_{0}-\varepsilon)}(\Omega),H^{(s_{1}+\delta)}(\Omega)]$,
with $\varepsilon,\delta>0$. This result follows immediately from
Theorem~\ref{th1}.
\end{remark}

Let a scale of Hilbert spaces $\{X_{s}\mid s\in\mathbb{R}\}$ be such that
the continuous embedding $X_{s_{1}}\hookrightarrow X_{s_{0}}$ holds whenever $s_{0}<s_{1}$. A
Hilbert space $H$ is called an interpolation space with respect to this scale if $H$
is an interpolation space with respect to a certain couple $[X_{s_{0}},X_{s_{1}}]$ with $s_{0}<s_{1}$.

\begin{theorem}\label{th2}
A Hilbert space $H$ is an interpolation space with respect to the Sobolev scale
$\{H^{(s)}(\Omega)\mid s\in\mathbb{R}\}$ if and only if $H=H^{\varphi}(\Omega)$ up to equivalence of norms for some $\varphi\in\mathrm{OR}$.
\end{theorem}

\begin{remark}\label{rem2}
There are functions $\varphi\in\mathrm{OR}$ for which $H^{\varphi}(\mathbb{R}^{n})$
is an intermediate but not an interpolation space for the couple
$[H^{(s_{0})}(\mathbb{R}^{n}),H^{(s_{1})}(\mathbb{R}^{n})]$. See Example \ref{ex3}.
\end{remark}

Theorem~\ref{th2} is a consequence of Theorem~\ref{th1} but is of independent interest. Both of them will be proved in Section~\ref{sec4}.

\section{Auxiliary abstract results}\label{sec3}

The exact Hilbert interpolation spaces with respect to a couple of Hilbert spaces were characterized (isometrically) by W.~F.~Donoghue \cite{Donoghue67} in 1967. V.~I.~Ovchinnikov \cite{Ovchinnikov84} later used this result to describe all Hilbert interpolation spaces, up to equivalence of norms. The theorems of Donoghue and Ovchinnikov are fairly deep, but it seems that they have not attract much attention to users of Hilbert spaces. Here we formulate some results of the interpolation theory in Hilbert spaces, including the Ovchinnikov theorem. It is sufficient to restrict ourselves to separable complex Hilbert spaces.

We say that an ordered couple $[X_{0},X_{1}]$ of Hilbert spaces $X_{0}$ and $X_{1}$ is admissible if these spaces are separable and the dense  continuous embedding $X_{1}\hookrightarrow X_{0}$ holds.

Let us recall the definition of the interpolation of Hilbert spaces with a function parameter. It is a natural generalization of the classical interpolation method of J.-L.~Lions and S.~G.~Krein (see, e.g., \cite[Chapter~1, Sec. 2, 5]{LionsMagenes72} and \cite[Chapter~3, Sec.~10]{KreinPetuninSemenov82}) to the case where a general enough function is used, instead of the number $\theta\in(0,\,1)$, as an interpolation parameter. The generalization appeared in C.~Foia\c{s} and J.-L.~Lions' paper \cite[Section 3.4]{FoiasLions61} and was then studied by several authors.

Following \cite[Sec.~2.1]{08MFAT1}, we denote by $\mathcal{B}$ the set of all Borel measurable functions $\psi:(0,\infty)\rightarrow(0,\infty)$ such that $\psi$ is bounded on each compact interval $[a,b]$ with $0<a<b<\infty$ and, moreover, $1/\psi$ is bounded on every set $[r,\infty)$ with $r>0$.

Let a function $\psi\in\mathcal{B}$ and an admissible couple of Hilbert spaces $X=[X_{0},X_{1}]$ be given. For $X$ there exists an isometric isomorphism $J: X_{1}\leftrightarrow X_{0}$ such that $J$ is a self-adjoint positive operator on $X_{0}$ with the domain $X_{1}$. The operator $J$ is called a generating operator for the couple~$X$. This operator is uniquely determined by $X$. (Indeed, assume that $J_{1}$ is another generating operator for~$X$. Then $J=J_{1}$ because $J$ and $J_{1}$ are positive and metrically equal.)

Using the spectral theorem, we can define an (unbounded) operator $\psi(J)$ in $X_{0}$ as the Borel function $\psi$ of the self-adjoint operator $J$. Let us denote by $[X_{0},X_{1}]_{\psi}$ or simply by $X_{\psi}$ the domain of the operator $\psi(J)$ endowed with the inner product $(u_{1},u_{2})_{X_{\psi}}:=(\psi(J)u_{1},\psi(J)u_{2})_{X_{0}}$ and the corresponding norm $\|u\|_{X_{\psi}}=\|\psi(J)u\|_{X_{0}}$. The space $X_{\psi}$ is Hilbert and separable.

A function $\psi\in\mathcal{B}$ is called an interpolation parameter if the
following condition is fulfilled for all admissible couples $X=[X_{0},X_{1}]$ and
$Y=[Y_{0},Y_{1}]$ of Hilbert spaces and for an arbitrary linear mapping $T$ given on
$X_{0}$: if the restriction of $T$ to $X_{j}$ is a bounded operator
$T:X_{j}\rightarrow Y_{j}$ for each $j\in\{0,1\}$, then the restriction of $T$ to
$X_{\psi}$ is also a bounded operator $T:X_{\psi}\rightarrow Y_{\psi}$.

If $\psi$ is an interpolation parameter, then we say that the Hilbert space
$X_{\psi}$ is obtained by the interpolation of $X$ with the function parameter
$\psi$. In this case, we have the continuous and dense embeddings $X_{1}\hookrightarrow X_{\psi}\hookrightarrow X_{0}$.

The classical result by J.-L.~Lions and S.~G.~Krein consists in that the power
function $\psi(t):=t^{\theta}$ is an interpolation parameter whenever $0<\theta<1$,
the exponent $\theta$ being regarded as a number parameter of the interpolation.

Let us describe the class of all interpolation parameters (in the sense of the above
definition).

Let a function $\psi:(0,\infty)\rightarrow(0,\infty)$ and a number $r\geq0$ be
given, then $\psi$ is called pseudoconcave on the semiaxis $(r,\infty)$ if there
exists a concave function $\psi_{1}:(r,\infty)\rightarrow(0,\infty)$ such that
$\psi(t)\asymp \psi_{1}(t)$ for $t>r$. The function $\psi$ is called pseudoconcave
in a neighbourhood of $+\infty$ if it is pseudoconcave on $(r,\infty)$, where $r$ is
a sufficiently large number.

\begin{proposition}\label{prop2} A function
$\psi\in\mathcal{B}$ is an interpolation parameter if and only if $\psi$ is
pseudoconcave in a neighbourhood of $+\infty$.
\end{proposition}

This fact follows from J.~Peetre's \cite{Peetre68} description of all interpolation functions for the weighted $\mathrm{L}_{p}(\mathbb{R}^{n})$-type spaces (see the monograph \cite[Theorem 5.4.4]{BerghLefstrem76} as well). A~proof of Proposition~\ref{prop2} is given
in, e.g., \cite[Section 2.7]{08MFAT1}. It is useful to note that
$\psi\in\mathcal{B}$ is pseudoconcave in a neighbourhood of $+\infty$ if and only if
$\psi$ is pseudoconcave on every set $(r,\infty)$ with $r>0$. This is evident in
view of the definition of the class $\mathcal{B}$.

Let $X=[X_{0},X_{1}]$ be an admissible couple of Hilbert spaces. V.~I.~Ovchinnikov
\cite[Theorem 11.4.1]{Ovchinnikov84} has described (up to equivalence of norms) all
the Hilbert spaces that are interpolation spaces with respect to $X$. In connection
with our considerations, his result can be restated as follows.

\begin{proposition}\label{prop3}
A Hilbert space $H$ is an interpolation space with respect to $X$ if and only if $H=X_{\psi}$ up to equivalence of norms for some function $\psi\in\mathcal{B}$ that is pseudoconcave in a neighbourhood of $+\infty$.
\end{proposition}

This proposition will play a key role in the proof of Theorem \ref{th1}, the main result of the paper. In this connection the following property of pseudoconcave functions will be of use.

\begin{proposition}\label{prop5}
Let a function $\psi\in\mathcal{B}$ and a number $r\geq0$ be given. The function
$\psi$ is pseudoconcave on $(r,\infty)$ if and only if there exists a number $c>0$
such that
\begin{equation}\label{eq9}
\frac{\psi(t)}{\psi(\tau)}\leq
c\,\max\left\{1,\,\frac{t}{\tau}\right\}\quad\mbox{for all}\quad t,\tau>r.
\end{equation}
\end{proposition}

In the $r=0$ case, this proposition was proved by J.~Peetre \cite{Peetre68} (see \cite[Theorem 5.4.4]{BerghLefstrem76} as well), the condition $\psi\in\mathcal{B}$ being
superfluous. In the $r>0$ case, the sufficiency is argued analogously; the necessity
is proved in \cite[Lemma 2.2]{08MFAT1} by reduction to the $r=0$ case.

We also need a reiteration theorem for interpolation with a function parameter
\cite[Theorems 2.1 and 2.3]{08MFAT1}.

\begin{proposition}\label{prop6}
Suppose that $f,g,\psi\in\mathcal{B}$ and that $f/g$ is bounded in a neighbourhood
of $+\infty$. Let $X$ be an admissible couple of Hilbert spaces. Then the couple
$[X_{f},X_{g}]$ is admissible, and $[X_{f},X_{g}]_{\psi}=X_{\omega}$ with equality
of norms. Here the function $\omega\in\mathcal{B}$ is given by the formula
$\omega(t):=f(t)\,\psi(g(t)/f(t))$ with $t>0$. Moreover, if $f,g$, and $\psi$ are
interpolation parameters, then $\omega$ is an interpolation parameter as well.
\end{proposition}

At the end of this section we will prove a result, which shows that
$H^{\varphi}(\mathbb{R}^{n})$ is well-defined.

\begin{proposition}\label{prop7}
Let $\varphi\in\mathrm{OR}$; then the function
$\mu(\xi):=\varphi(\langle\xi\rangle)$ of $\xi\in\mathbb{R}^{n}$ satisfies
$\eqref{eq2}$, i.e. $\mu$ is a weight function.
\end{proposition}

\begin{proof}
Let $\xi,\,\eta\in\mathbb{R}^{n}$. By taking squares it is easy to check the inequality $|\langle\xi\rangle-\nobreak\langle\eta\rangle|\leq|\,|\xi|-|\eta|\,|$.
Therefore, in the $\langle\xi\rangle\geq\langle\eta\rangle$ case, we have
\begin{equation*}
\frac{\langle\xi\rangle}{\langle\eta\rangle}=
1+\frac{\langle\xi\rangle-\langle\eta\rangle}{\langle\eta\rangle}\leq
1+|\xi|-|\eta|\leq1+|\xi-\eta|.
\end{equation*}
Then, by Proposition~\ref{prop1} (iii), we may write
\begin{equation*}
\frac{\varphi(\langle\xi\rangle)}{\varphi(\langle\eta\rangle)}\leq
c\,\left(\frac{\langle\xi\rangle}{\langle\eta\rangle}\right)^{s_{1}}\leq
c\,\left(1+|\xi-\eta|\right)^{\max\{0,s_{1}\}}.
\end{equation*}
Besides, if $\langle\eta\rangle\geq\langle\xi\rangle$, then
\begin{equation*}
\frac{\varphi(\langle\xi\rangle)}{\varphi(\langle\eta\rangle)}\leq
c\,\left(\frac{\langle\xi\rangle}{\langle\eta\rangle}\right)^{s_{0}}=
c\,\left(\frac{\langle\eta\rangle}{\langle\xi\rangle}\right)^{-s_{0}}\leq
c\,\left(1+|\xi-\eta|\right)^{\max\{0,-s_{0}\}}.
\end{equation*}
Thus
\begin{equation*}
\frac{\varphi(\langle\xi\rangle)}{\varphi(\langle\eta\rangle)}\leq
c\,(1+|\xi-\eta|)^{\ell}\quad\mbox{for all}\quad\xi,\eta\in\mathbb{R}^{n},
\end{equation*}
with $\ell:=\max\{0,-s_{0},s_{1}\}$. This yields \eqref{eq2} for a certain constant
$c\geq1$.\qed
\end{proof}

\begin{remark}\label{rem3.6}
There are Borel measurable functions
$\varphi:\nobreak[1,\infty)\rightarrow(0,\infty)$ such that
$\varphi\notin\mathrm{OR}$ but $\mu(\xi):=\varphi(\langle\xi\rangle)$ is a weight
function of $\xi\in\mathbb{R}^{n}$ (see Example \ref{ex2}).
\end{remark}

\section{Proof of the main result}\label{sec4}

Beforehand, we prove two theorems. In the first one we describe the space which is a
result of the interpolation between Sobolev spaces provided the function
interpolation parameter is used.

\begin{theorem}\label{lem1}
Let numbers $s_{0},s_{1}\in\mathbb{R}$ be such that $s_{0}<s_{1}$, and let a function $\psi\in\mathcal{B}$ be an interpolation parameter. Set
\begin{equation}\label{eq10}
\varphi(t):=t^{s_{0}}\,\psi(t^{s_{1}-s_{0}})\quad\mbox{for}\quad t\geq1.
\end{equation}
Then $\varphi\in\mathrm{OR}$ and
\begin{equation}\label{eq12}
[H^{(s_{0})}(\Omega),H^{(s_{1})}(\Omega)]_{\psi}=H^{\varphi}(\Omega)
\end{equation}
up to equivalence of norms. If $\Omega=\mathbb{R}^{n}$, then \eqref{eq12} holds with
equality of norms.
\end{theorem}

\begin{proof}
First we prove that $\varphi\in\mathrm{OR}$. By definition, the function $\varphi$
is Borel measurable on $[1,\infty)$. Let us prove that $\varphi$ satisfies
\eqref{eq3}. Since the function $\psi\in\mathcal{B}$ is an interpolation parameter,
it is pseudoconcave on $(r,\infty)$, where we may put $r=1$ (see
Proposition~\ref{prop2} and the comment to it). Hence, according to
Proposition~\ref{prop5}, we may write
\begin{gather}\label{eq13}
\frac{\varphi(\lambda t)}{\varphi(t)}= \lambda^{s_{0}}\,\frac{\psi((\lambda
t)^{s_{1}-s_{0}})}{\psi(t^{s_{1}-s_{0}})}\leq
\lambda^{s_{0}}\,c\,\max\{1,\lambda^{s_{1}-s_{0}}\}=c\,\lambda^{s_{1}},\\
\frac{\varphi(t)}{\varphi(\lambda t)}=
\lambda^{-s_{0}}\,\frac{\psi(t^{s_{1}-s_{0}})}{\psi((\lambda t)^{s_{1}-s_{0}})}\leq
\lambda^{-s_{0}}\,c\,\max\{1,\lambda^{s_{0}-s_{1}}\}=
c\,\lambda^{-s_{0}}\label{eq14}
\end{gather}
for arbitrary $t\geq1$, $\lambda\geq1$, and a certain number $c>0$ that is independent of $t$ and $\lambda$. This shows \eqref{eq5}; therefore, $\varphi$ satisfies \eqref{eq3} with $a=2$ and, hence, belongs to $\mathrm{OR}$.

Let us prove that
\begin{equation}\label{eq11}
[H^{(s_{0})}(\mathbb{R}^{n}),H^{(s_{1})}(\mathbb{R}^{n})]_{\psi}=
H^{\varphi}(\mathbb{R}^{n})
\end{equation}
with equality of norms. The couple of Sobolev spaces
$[H^{(s_{0})}(\mathbb{R}^{n}),H^{(s_{1})}(\mathbb{R}^{n})]$ is admissible. Let $J$
denote the pseudodifferential operator whose symbol is the function
$\langle\xi\rangle^{s_{1}-s_{0}}$ of $\xi\in\mathbb{R}^{n}$. Then $J$ is a
generating operator for this couple. Using the Fourier transform
$\mathcal{F}:H^{(s_{0})}(\mathbb{R}^{n})\leftrightarrow
\mathrm{L}_{2}(\mathbb{R}^{n},\langle\xi\rangle^{2s_{0}}\,d\xi)$, we reduce $J$ to
an operator of multiplication by the function $\langle\xi\rangle^{s_{1}-s_{0}}$.
Hence, $\psi(J)$ is reduced to an operator of multiplication by the function
$\psi(\langle\xi\rangle^{s_{1}-s_{0}})=
\langle\xi\rangle^{-s_{0}}\varphi(\langle\xi\rangle)$. Therefore, we may write the
following:
\begin{align*}
\|u\|_{[H^{(s_{0})}(\mathbb{R}^{n}),H^{(s_{1})}(\mathbb{R}^{n})]_{\psi}}
^{2}&=\|\psi(J)u\|_{H^{(s_{0})}(\mathbb{R}^{n})}^{2}\\
&=\int\limits_{\mathbb{R}^{n}}
|(\widehat{\psi(J)u})(\xi)|^{2}\,\langle\xi\rangle^{2s_{0}}\,d\xi\\
&=\int\limits_{\mathbb{R}^{n}}|\,\psi(\langle\xi\rangle^{s_{1}-s_{0}})\,
\widehat{u}(\xi)|^{2}\,\langle\xi\rangle^{2s_{0}}\,d\xi\\
&=\int\limits_{\mathbb{R}^{n}}\varphi^{2}(\langle\xi\rangle)\:|\widehat{u}(\xi)|^{2}\,
d\xi=\|u\|_{H^{\varphi}(\mathbb{R}^{n})}^{2}
\end{align*}
for every $u\in C^{\infty}_{0}(\mathbb{R}^{n})$. This implies the equality of spaces
\eqref{eq11} as $C^{\infty}_{0}(\mathbb{R}^{n})$ is dense in both of them. (Note
that $C^{\infty}_{0}(\mathbb{R}^{n})$ is dense in the interpolation space
$[H^{(s_{0})}(\mathbb{R}^{n}),H^{(s_{1})}(\mathbb{R}^{n})]_{\psi}$ because
$C^{\infty}_{0}(\mathbb{R}^{n})$ is dense in the Sobolev space
$H^{(s_{1})}(\mathbb{R}^{n})$, which is embedded continuously and densely in the
interpolation space.)

Now formula \eqref{eq12} will be deduced from \eqref{eq11} in the remaining cases considered for $\Omega$, namely, where $\Omega$ is either an open half-space or a bounded domain in $\mathbb{R}^{n}$ with Lipschits boundary (cf. \cite[Sec. 1.1.18]{Triebel06}, where classical interpolation methods are used). Note that the couple of Sobolev spaces in \eqref{eq12} is admissible. Let $R_{\Omega}$ stand for the
operator that restricts distributions $u\in\mathcal{S}'(\mathbb{R}^{n})$ to $\Omega$. We have the surjective bounded operators
\begin{gather}\label{eq15}
R_{\Omega}:\,H^{(s)}(\mathbb{R}^{n})\rightarrow H^{(s)}(\Omega),\quad s\in\mathbb{R},\\
R_{\Omega}:\,H^{\varphi}(\mathbb{R}^{n})\rightarrow H^{\varphi}(\Omega).
\label{eq16}
\end{gather}
Applying the interpolation with the parameter $\psi$, we infer, by \eqref{eq11},
that the boundedness of the operators \eqref{eq15}, with $s\in\{s_{0},s_{1}\}$,
implies boundedness of the operator
\begin{equation*}
R_{\Omega}:\,H^{\varphi}(\mathbb{R}^{n})=
[H^{(s_{0})}(\mathbb{R}^{n}),H^{(s_{1})}(\mathbb{R}^{n})]_{\psi}\rightarrow
[H^{(s_{0})}(\Omega),H^{(s_{1})}(\Omega)]_{\psi}.
\end{equation*}
Hence, since the operator \eqref{eq16} is surjective, we have the continuous embedding
\begin{equation}\label{eq17}
H^{\varphi}(\Omega)\hookrightarrow
[H^{(s_{0})}(\Omega),H^{(s_{1})}(\Omega)]_{\psi}.
\end{equation}

Let us prove the inverse inclusion and its continuity. We need to use a
linear mapping, say $T$, that extends every distribution
$$
u\in\bigcup_{s\in\mathbb{R}}H^{(s)}(\Omega)
$$
to $\mathbb{R}^{n}$ and defines a bounded operator
\begin{equation}\label{eq18}
T:\ H^{(s)}(\Omega)\rightarrow H^{(s)}(\mathbb{R}^{n})\quad\mbox{for each}\quad s\in\mathbb{R}.
\end{equation}
This mapping is constructed by R.~Seeley \cite{Seeley64} in the case where $\Omega$ is a half-space, and by V.~S.~Rychkov \cite{Rychkov99} in the case where $\Omega$ is a bounded domain with Lipschitz boundary. Consider the operators \eqref{eq18} for $s=s_{0}$ and $s=s_{1}$. Since $\psi$ is an interpolation
parameter, their boundedness and formula \eqref{eq11} yield boundedness of the
operator
\begin{equation}\label{eq19}
T:\,[H^{(s_{0})}(\Omega),H^{(s_{1})}(\Omega)]_{\psi}\rightarrow
[H^{(s_{0})}(\mathbb{R}^{n}),H^{(s_{1})}(\mathbb{R}^{n})]_{\psi}=
H^{\varphi}(\mathbb{R}^{n}).
\end{equation}
The product of the bounded operators \eqref{eq16} and \eqref{eq19} gives us the
bounded identity operator
\begin{equation*}
I=R_{\Omega}T:\,[H^{(s_{0})}(\Omega),H^{(s_{1})}(\Omega)]_{\psi}\rightarrow
H^{\varphi}(\Omega).
\end{equation*}
Thus, together with the continuous embedding \eqref{eq17}, we have its continuous
inverse; i.e., \eqref{eq12} holds up to equivalence of norms. \qed
\end{proof}

\begin{theorem}\label{lem2}
Let $s_{0},s_{1}\in\mathbb{R}$, with $s_{0}<s_{1}$, and let $\psi\in\mathcal{B}$.
Suppose that $\varphi$ is defined by \eqref{eq10}. Then $\psi$ is an interpolation
parameter if and only if $\varphi$ satisfies \eqref{eq5} with some number $c\geq1$
that is independent of $t$ and $\lambda$.
\end{theorem}

\begin{proof}
If $\psi$ is an interpolation parameter, then, as we have proved above, the function
$\varphi$ satisfies \eqref{eq13} and \eqref{eq14} for arbitrary $t\geq1$ and
$\lambda\geq1$, i.e. \eqref{eq5} is fulfilled.

Conversely, suppose that $\varphi$ satisfies \eqref{eq5}. Let us prove the
inequality \eqref{eq9} for $\psi$. Considering arbitrary $t\geq\tau\geq1$ and
applying the right-hand inequality in \eqref{eq5}, we may write
\begin{align*}
\frac{\psi(t)}{\psi(\tau)}&=
\frac{t^{-s_{0}/(s_{1}-s_{0})}\,\varphi(t^{1/(s_{1}-s_{0})})}
{\tau^{-s_{0}/(s_{1}-s_{0})}\,\varphi(\tau^{1/(s_{1}-s_{0})})}
\\&\leq\lambda^{-s_{0}/(s_{1}-s_{0})}\,c\,\lambda^{s_{1}/(s_{1}-s_{0})}=c\,\lambda=
c\,\max\left\{1,\,\frac{t}{\tau}\right\};
\end{align*}
here $\lambda:=t/\tau\geq1$, whereas the number $c>0$ does not depend on $t$ and
$\tau$. Analogously, considering any $\tau\geq t\geq1$ and applying the left-hand
inequality in \eqref{eq5}, we may write
\begin{equation*} \frac{\psi(t)}{\psi(\tau)}\leq
\lambda^{s_{0}/(s_{1}-s_{0})}\,c\,\lambda^{-s_{0}/(s_{1}-s_{0})}=c=
c\,\max\left\{1,\,\frac{t}{\tau}\right\},
\end{equation*}
with $\lambda:=\tau/t\geq1$. Thus, the inequality \eqref{eq9} holds for $r=1$.
Hence, we conclude, by Propositions \ref{prop2} and \ref{prop5}, that $\psi$ is an
interpolation parameter. \qed
\end{proof}


\textsl{Proof of Theorem~$\ref{th1}$. Necessity.} Let a Hilbert space $H$ be an interpolation space with respect to the couple of Sobolev
spaces $[H^{(s_{0})}(\Omega),H^{(s_{1})}(\Omega)]$. Then, by Propositions
\ref{prop3}, \ref{prop2}, and Theorem~\ref{lem1}, we conclude that
\begin{equation*}
H=[H^{(s_{0})}(\Omega),H^{(s_{1})}(\Omega)]_{\psi}=H^{\varphi}(\Omega)
\end{equation*}
up to equivalence of norms. Here $\psi\in\mathcal{B}$ is a certain interpolation
function parameter, whereas $\varphi$ is defined by \eqref{eq10}. The function $\varphi$ satisfies \eqref{eq5} in view of Theorem~\ref{lem2} and, hence, belongs to $\mathrm{OR}$. The necessity is proved.

\medskip

\textsl{Sufficiency.} Let a function parameter $\varphi\in\mathrm{OR}$ satisfy condition \eqref{eq5}. Suppose that a Hilbert space $H$
coincides with $H^{\varphi}(\Omega)$ up to equivalence of norms. Starting with
$\varphi$, we construct a Borel measurable function $\psi$ such that \eqref{eq10}
holds. Namely, we set
\begin{equation}\label{eq20}
\psi(\tau):=
\begin{cases}
\;\tau^{-s_{0}/(s_{1}-s_{0})}\,\varphi(\tau^{1/(s_{1}-s_{0})}) &\text{for}\quad\tau\geq1, \\
\;\varphi(1) & \text{for}\quad0<\tau<1.
\end{cases}
\end{equation}
Note that $\psi\in\mathcal{B}$ in view of Proposition \ref{prop1} (i) and condition \eqref{eq5}. (This condition written for $t=1$ yields $\psi(\tau)/\varphi(1)\geq c^{-1}$ for every $\tau\geq1$.) By Theorem~\ref{lem2}, the function $\psi$ is an interpolation parameter. Therefore, applying Theorem~\ref{lem1}, we conclude that
\begin{equation}\label{eq21}
H=H^{\varphi}(\Omega)=[H^{(s_{0})}(\Omega),H^{(s_{1})}(\Omega)]_{\psi}
\end{equation}
up to equivalence of norms. Hence $H$ is an interpolation space with respect to the
couple $[H^{(s_{0})}(\Omega),H^{(s_{1})}(\Omega)]$. The sufficiency is proved. \qed

\medskip

The following proof is simple, but we give it for the sake of completeness.

\medskip

\textsl{Proof of Theorem~$\ref{th2}$. Necessity.} Let a Hilbert space $H$
be an interpolation space with respect to the Sobolev scale
$\{H^{(s)}(\Omega):s\in\mathbb{R}\}$. Then $H$ is an interpolation space with
respect to a certain couple $[H^{(s_{0})}(\Omega),H^{(s_{1})}(\Omega)]$, with
$-\infty<s_{0}<s_{1}<\infty$. Hence, by Theorem~\ref{th1}, we conclude that
$H=H^{\varphi}(\Omega)$ up to equivalence of norms for some $\varphi\in\mathrm{OR}$.
The necessity is proved.

\medskip

\textsl{Sufficiency.} Let $\varphi\in\mathrm{OR}$, and let a Hilbert space $H$ coincide with $H^{\varphi}(\Omega)$ up to equivalence of norms. Choose numbers $s_{0},s_{1}\in\mathbb{R}$ such that $s_{0}<\nobreak\sigma_{0}(\varphi)$ and $\sigma_{1}(\varphi)<\nobreak s_{1}$. Then condition \eqref{eq5} is satisfied. According to Theorem~\ref{th1}, $H$ is an interpolation space with respect to the couple $[H^{(s_{0})}(\Omega),H^{(s_{1})}(\Omega)]$ and, consequently, with respect to the Sobolev scale $\{H^{(s)}(\Omega)\mid
s\in\mathbb{R}\}$. The sufficiency is proved. \qed

\section{The extended Sobolev scale}\label{sec5}

Here we expound two important interpolation properties of the extended Sobolev scale
\begin{equation}\label{eq22}
\{H^{\varphi}(\Omega)\mid\varphi\in\mathrm{OR}\}
\end{equation}
(see Theorems \ref{th3} and \ref{th4} below). Specifically, they explain why we have chosen such a name for the class of H\"ormander spaces \eqref{eq22}.

Owing to the first property, every space in \eqref{eq22} can be obtained by the
interpolation of an appropriate couple of Sobolev spaces provided that a certain
interpolation function parameter is used. Just this result can be a basis for
further applications of the scale \eqref{eq22} in various problems.

\begin{theorem}\label{th3}
Let $\varphi\in\mathrm{OR}$. Choose numbers $s_{0},s_{1}\in\mathbb{R}$ such that
$s_{0}<\sigma_{0}(\varphi)<\sigma_{1}(\varphi)<s_{1}$ and define a function $\psi$
by formula \eqref{eq20}. Then $\psi\in\mathcal{B}$ is an interpolation parameter,
and
\begin{equation}\label{eq23}
[H^{(s_{0})}(\Omega),H^{(s_{1})}(\Omega)]_{\psi}=H^{\varphi}(\Omega)
\end{equation}
up to equivalence of norms. If $\Omega=\mathbb{R}^{n}$, then \eqref{eq23} holds with
equality of norms.
\end{theorem}

\begin{proof}
Note that $\psi$ belongs to $\mathcal{B}$ and is an interpolation parameter; this
has been demonstrated in the proof of Theorem~\ref{th1} (sufficiency). Now, since
$\varphi$ satisfies \eqref{eq10}, Theorem~\ref{th3} is a direct corollary of Theorem~\ref{lem1}. \qed
\end{proof}

The second property reveals that the class of spaces \eqref{eq22} is
closed with respect to interpolation with a function parameter.

\begin{theorem}\label{th4}
Let functions $\varphi_{0},\varphi_{1}\in\mathrm{OR}$ and $\psi\in\mathcal{B}$ be
given. Suppose that $\varphi_{0}/\varphi_{1}$ is bounded in a neighbourhood of
$+\infty$ and that $\psi$ is an interpolation parameter. Set
\begin{equation*}
\varphi(t):=\varphi_{0}(t)\,\psi
\left(\frac{\varphi_{1}(t)}{\varphi_{0}(t)}\right)\quad\mbox{for}\quad t\geq1.
\end{equation*}
Then $\varphi\in\mathrm{OR}$, and
\begin{equation}\label{eq24}
[H^{\varphi_{0}}(\Omega),H^{\varphi_{1}}(\Omega)]_{\psi}= H^{\varphi}(\Omega)
\end{equation}
up to equivalence of norms. If $\Omega=\mathbb{R}^{n}$, then \eqref{eq24} holds with
equality of norms.
\end{theorem}

\begin{proof} Let us deduce \eqref{eq24} from \eqref{eq23} with the help of
the reiterated interpolation with a function parameter. Choose numbers
$s_{0},s_{1}\in\mathbb{R}$ such that $s_{0}<\sigma_{0}(\varphi_{j})$ and
$s_{1}>\sigma_{1}(\varphi_{j})$ for each $j\in\{0,1\}$. By Theorem~\ref{th3}, we
have
\begin{equation*}
[H^{(s_{0})}(\Omega),H^{(s_{1})}(\Omega)]_{\psi_{j}}=H^{\varphi_{j}}(\Omega)\quad
\mbox{for each}\quad j\in\{0,1\}.
\end{equation*}
Here the function $\psi_{j}\in\mathcal{B}$ is the interpolation parameter defined by
\eqref{eq20} with $\varphi=\varphi_{j}$. \textcolor{red}{Note} that $\psi_{0}/\psi_{1}$ is bounded in
a neighbourhood of $+\infty$. According to Proposition~\ref{prop6} and
Theorem~\ref{lem1}, we may write
\begin{align*}
[H^{\varphi_{0}}(\Omega),H^{\varphi_{1}}(\Omega)]_{\psi}&=
\bigl[\,[H^{(s_{0})}(\Omega),H^{(s_{1})}(\Omega)]_{\psi_{0}},
[H^{(s_{0})}(\Omega),H^{(s_{1})}(\Omega)]_{\psi_{1}}\,\bigr]_{\psi}\\
&=[H^{(s_{0})}(\Omega),H^{(s_{1})}(\Omega)]_{\omega}=H^{\varphi}(\Omega).
\end{align*}
Here the interpolation parameter $\omega\in\mathcal{B}$ satisfies the equality
\begin{align*}
\omega(\tau):&=\psi_{0}(\tau)\,\psi\Bigl(\frac{\psi_{1}(\tau)}{\psi_{0}(\tau)}\Bigr)\\&=
\tau^{-s_{0}/(s_{1}-s_{0})}\,\varphi_{0}(\tau^{1/(s_{1}-s_{0})})\,
\psi\Bigl(\frac{\varphi_{1}(\tau^{1/(s_{1}-s_{0})})}
{\varphi_{0}(\tau^{1/(s_{1}-s_{0})})}\Bigr)
\end{align*}
for $\tau\geq1$. Hence, the function
\begin{equation*}
\varphi(t):=\varphi_{0}(t)\,\psi\Bigl(\frac{\varphi_{1}(t)}{\varphi_{0}(t)}\Bigr)=
t^{s_{0}}\,\omega(t^{s_{1}-s_{0}}),\quad\quad t\geq1,
\end{equation*}
belongs to the class $\mathrm{OR}$ according to Theorem~\ref{lem1}. The equality of spaces is written up to equivalence of norms, with the equivalence becoming equality in the case of $\Omega=\mathbb{R}^{n}$. \qed
\end{proof}

Thus the extended Sobolev scale \eqref{eq22} is the final extension of the scale
$\{H^{(s)}(\Omega)\mid s\in\mathbb{R}\}$ by the interpolation within the category of
Hilbert spaces.

\section{Examples}\label{sec6}

Theorem \ref{th1} allows us to construct three explicit examples of Hilbert spaces that are intermediate for given couples of Sobolev spaces but differ considerably on their interpolation properties.

Namely, for the couples
\begin{equation}\label{f6.1}
[H^{(s_{0})}(\Omega),H^{(s_{1})}(\Omega)],\quad\mbox{with}\quad s_{0}\leq0<s_{1},
\end{equation}
we construct examples of intermediate spaces $H^{\varphi}(\Omega)$ that each of the following possibilities is realized. For every couple \eqref{f6.1}, the space $H^{\varphi}(\Omega)$
\begin{itemize}
\item [(i)] is an interpolation space;
\item [(ii)] is not an interpolation space;
\item [(iii)] is not an interpolation space provided that $s_{0}=0$, but is an interpolation space provided that $s_{0}<0$.
\end{itemize}

\begin{example}\label{ex1}
First we give an example of a Hilbert space $H^{\varphi}(\Omega)$ which is an interpolation (and intermediate) space for every couple \eqref{f6.1}. We obtain the required space $H^{\varphi}(\Omega)$ if we
take $\varphi(t):=1+\log t$ for $t\geq1$. This follows from Theorem \ref{th1}
because $\sigma_{0}(\varphi)=0=\sigma_{1}(\varphi)$ and the supremum in
$\eqref{eq6}$ is evidently attained.
\end{example}

\begin{example}\label{ex2}
Now we give an example of a Hilbert space $H^{\varphi}(\Omega)$ such that, for every couple \eqref{f6.1}, $H^{\varphi}(\Omega)$ is an intermediate space but is not an interpolation space. Let us put
$$
\varphi(t):=1+(1+\sin\log t)\log t\quad\mbox{for}\quad t\geq1
$$
and show that the space $H^{\varphi}(\Omega)$ is as required.

The corresponding radial function $\mu(\xi):=\varphi(\langle\xi\rangle)$ of
$\xi\in\mathbb{R}^{n}$ is a weight function, i.e. \eqref{eq2} is fulfilled. Indeed,
given $t,\tau\geq1$, we may write
\begin{equation}\label{f6.2}
\frac{\varphi(t)}{\varphi(\tau)}\leq
1+\frac{|\varphi(t)-\varphi(\tau)|}{\varphi(\tau)}\leq
1+|\varphi(t)-\varphi(\tau)|\leq 1+3\,|t-\tau|.
\end{equation}
Here we use the estimate $|\varphi'(\theta)|\leq3$ for every $\theta\geq1$, which is
directly verified. It follows from \eqref{f6.2} that the function
$\mu(\xi):=\varphi(\langle\xi\rangle)$ satisfies \eqref{eq2} with $c=3$ and $\ell=1$.
Thus, the H\"ormander space $H^{\varphi}(\Omega)$ is well-defined.

Choose numbers $s_{0}\leq0$ and $s_{1}>0$ arbitrarily. Since
$1\leq\varphi(t)\leq1+2\log t$ for $t\geq1$, we have the continuous embeddings $H^{(s_{1})}(\Omega)\hookrightarrow
H^{\varphi}(\Omega)\hookrightarrow H^{(s_{0})}(\Omega)$ by virtue
of \eqref{f2.2}. Note also that $\varphi\notin\mathrm{OR}$. Indeed, putting
$\lambda:=\exp(\pi/2)$ and $t_{k}:=\exp(2\pi k-\pi/2)$ with $k=1,2,3,\ldots$, we
have $\varphi(\lambda t_{k})/\varphi(t_{k})=1+2\pi k\rightarrow\infty$ as
$k\rightarrow\infty$, contrary to the property \eqref{eq5} of the class OR.
Therefore, by Theorem \ref{th1}, $H^{\varphi}(\Omega)$ is not an inter\-po\-la\-tion
space for the couple \eqref{f6.1}.

Let us argue the latter conclusion in more detail. If $H^{\varphi}(\Omega)$ were an
inter\-po\-la\-tion space for this couple, then
$H^{\varphi}(\Omega)=H^{\varphi_{1}}(\Omega)$ up to equivalence of norms for some
$\varphi_{1}\in\mathrm{OR}$ by Theorem \ref{th1}. Hence, $\varphi\asymp\varphi_{1}$ on $[1,\infty)$ (see Remark~\ref{rem}). This contradicts $\varphi\notin\mathrm{OR}$.
\end{example}

\begin{example}\label{ex3}
Finally, we give an example of a Hilbert space $H^{\varphi}(\Omega)$ such that, for each couple $[H^{(0)}(\Omega),H^{(s_{1})}(\Omega)]$ with $s_{1}>0$, $H^{\varphi}(\Omega)$ is an intermediate space but is not an interpolation space and, moreover, $H^{\varphi}(\Omega)$ is an interpolation space for every couple $[H^{(s_{0})}(\Omega),H^{(s_{1})}(\Omega)]$ with $s_{0}<0<s_{1}$.

Set $h(t):=(\log t)^{-1/2}\sin(\log^{1/4}t)$ and define the function
\begin{equation*}
\varphi(t):=
\begin{cases}
\;t^{h(t)}+\log t & \text{if}\;\;t\geq3, \\
\;1 & \text{if}\;\;0<t<3.
\end{cases}
\end{equation*}
Let us show that the space $H^{\varphi}(\Omega)$ is as required.

Evidently, $\varphi\in\mathcal{B}$. A simple calculation shows that
$t\varphi'(t)/\varphi(t)\rightarrow\nobreak0$ as $t\rightarrow\infty$. Hence
\cite[Section 1.2]{Seneta76}, the function $\varphi$ is slowly varying at infinity
in the sense of J.~Karamata; i.e.,
\begin{equation}\label{f6.3}
\lim_{t\rightarrow\infty}\,\frac{\varphi(\lambda\,t)}{\varphi(t)}=1\quad\mbox{for
each}\quad\lambda>0.
\end{equation}
Owing to Uniform Convergence Theorem \cite[Section 1.2, Theorem 1.1]{Seneta76}, the convergence in \eqref{f6.3} is uniform on every compact $\lambda$-set in $(0,\infty)$. Therefore $\varphi\in\mathrm{OR}$ and, moreover,
$\sigma_{0}(\varphi)=\sigma_{1}(\varphi)=\nobreak0$ (see \cite[Section
2.1]{BinghamGoldieTeugels89}). Thus the space $H^{\varphi}(\Omega)$ is well-defined.

Since both the functions $1/\varphi(t)$ and $\varphi(t)/t^{s_{1}}$ are bounded on $[1,\infty)$ for every $s_{1}>0$, the continuous embeddings $H^{(s_{1})}(\Omega)\hookrightarrow
H^{\varphi}(\Omega)\hookrightarrow H^{(0)}(\Omega)$ hold.

Next, we show that $\varphi$ is not pseudoconcave on $(r,\infty)$ whenever $r>0$.
Consider the sequences of numbers $t_{k}:=\exp((2\pi k+\pi/2)^{4})$ and
$s_{k}:=\exp((2\pi k+\pi)^{4})$, with $k=1,\,2,\,3,\ldots$. Straightforward calculations yield
$h(t_{k})=(2\pi k+\pi/2)^{-2}$ and $h(s_{k})=0$; hence,
\begin{equation*}
\log\varphi(t_{k})\geq h(t_{k})\log t_{k}=\Bigl(2\pi
k+\frac{\pi}{2}\Bigr)^{2}\quad\mbox{and}\quad\varphi(s_{k})=1+(2\pi k+\pi)^{4}.
\end{equation*}
Therefore,
\begin{equation*}
\frac{\varphi(t_{k})}{\varphi(s_{k})}\geq\frac{\exp((2\pi k+\pi/2)^{2})}{(1+(2\pi
k+\pi)^{4})}\rightarrow\infty\quad\mbox{as}\quad k\rightarrow \infty.
\end{equation*}
But $t_{k}<s_{k}$, then, by Proposition~\ref{prop5}, the function $\varphi$ is not
pseudoconcave on $(r,\infty)$ whenever $r>0$.

Applying this fact, we will prove that $H^{\varphi}(\Omega)$ is not an interpolation
space for the couple $[H^{(0)}(\Omega),H^{(1)}(\Omega)]$. Suppose the contrary;
then, by Proposition~\ref{prop3}, we may write
$H^{\varphi}(\Omega)=[H^{(0)}(\Omega),H^{(1)}(\Omega)]_{\psi}$ (up to equivalence of norms) for some function parameter $\psi\in\mathcal{B}$ which is pseudoconcave
in a neighbourhood of~$+\infty$. Hence, $H^{\varphi}(\Omega)=H^{\psi}(\Omega)$ by
Theorem~\ref{th4} applied to the functions $\varphi_{0}(t)\equiv t^{0}$ and $\varphi_{1}(t)\equiv t^{1}$. As we have noted in Remark~\ref{rem}, the equality $H^{\varphi}(\Omega)=H^{\psi}(\Omega)$ is equivalent to the property $\varphi\asymp\psi$ on $[1,\infty)$. Hence, $\varphi$ is pseudoconcave in a neighbourhood of~$+\infty$, a contradiction.

Thus, $H^{\varphi}(\Omega)$ is not an interpolation space for the couple
$$
\bigl[H^{(0)}(\Omega),H^{(1)}(\Omega)\bigr].
$$
It follows from this and Theorem \ref{th1} (see Remark~\ref{rem1}) that the supremum in $\eqref{eq6}$ is not attained (recall that $\sigma_{0}(\varphi)=0=\sigma_{1}(\varphi)$). Hence, by the
same theorem, $H^{\varphi}(\Omega)$ is not an interpolation space for the couple $[H^{(0)}(\Omega),H^{(s_{1})}(\Omega)]$ whenever $s_{1}>0$, but is an interpolation space for each couple \eqref{f6.1} provided that
$s_{0}<0$.
\end{example}

\end{document}